\documentclass[12pt]{amsart}

\usepackage{verbatim,amssymb,latexsym,amscd}
\usepackage{times}
\usepackage[all]{xy}
\usepackage[colorlinks,linkcolor=blue,citecolor=blue,urlcolor=red]{hyperref}
\usepackage[T1]{fontenc}

\newcommand{\Z}{\mathbf{Z}}

\newcommand{\sC}{\mathcal{C}}

\newcommand{\sF}{\mathcal{F}}

\newcommand{\sI}{\mathcal{I}}

\newcommand{\sS}{\mathcal{S}}

\newcommand{\uT}{\underline{T}}

\newcommand{\Tor}{\operatorname{Tor}}
\newcommand{\Hom}{\operatorname{Hom}}

\newcommand{\Ext}{\operatorname{Ext}}

\newcommand{\Ab}{\operatorname{\bf Ab}}

\newcommand{\tors}{{\operatorname{tors}}}

\newcommand{\by}{\xrightarrow}

\newcommand{\iso}{\by{\sim}}

\renewcommand{\lim}{\varprojlim}

\renewcommand{\qed}{\hfill $\Box$\medskip}

\newcommand{\oo}{\operatorname{\overset{L}{\otimes}}}

\renewcommand{\phi}{\varphi}
\renewcommand{\epsilon}{\varepsilon}

\newcounter{spec}
\newenvironment{thlist}{\begin{list}{\rm{(\roman{spec})}}%
{\usecounter{spec}\labelwidth=20pt\itemindent=0pt\labelsep=10pt}}%
{\end{list}}%

\newtheorem{thm}{Theorem}

\newtheorem{cor}[thm]{Corollary}

\theoremstyle{definition}

\begin{document}
\title{A universal coefficient theorem for sheaf cohomology}
\author{Bruno Kahn}
\address{CNRS, Sorbonne Université and Université Paris Cité, IMJ-PRG\\ Case 247\\4 place
Jussieu\\75252 Paris Cedex 05\\France}
\email{bruno.kahn@imj-prg.fr}
\begin{abstract} We give a sheaf-theoretic version of the universal coefficient theorem.
\end{abstract}
\date{\today}
\maketitle


The universal coefficient theorem, originally due to Eilenberg and Mac Lane \cite{eml}, has known several variants and improvements. To my surprise, I didn't find a version for sheaf cohomology in the literature. Maybe the reason is simply that nobody had any need for it, which is my case now.

The question is this: let $\sF$ be a sheaf of abelian groups over a site $X$, and let $A$ be an abelian group. Compare the groups $H^r(X,\sF)\otimes A$ and $H^r(X,\sF\otimes A)$.

In topology, universal coefficient theorems are usually expressed in terms of $\otimes$ and $\Tor$ for homology, and in terms of $\Hom$ and $\Ext$ for cohomology, e.g. \cite[\S 3.6]{weibel}. The basic reason is that the functor $\otimes$ is right exact, while $\Hom$ is left exact. Here we want to mix $\otimes$ with cohomology, which makes things more complicated; in particular, there is \emph{a priori} no obvious comparison map between the above two groups. Another complication comes from the possible torsion in $\sF$. Yet another more subtle issue is that sheaf cohomology does not necessarily commute with filtered colimits.

As could be expected, a good solution is obtained by using derived categories. This will give a universal coefficient theorem in the form of a projection formula. Namely, let $f:X\to Y$ be a morphism of sites and let $\sS(X)$ (resp. $\sS(Y)$) be the category of sheaves of abelian groups on $X$ (resp. $Y$). The inverse image functor $f^*:D(\sS(Y))\to D(\sS(X))$ has the right adjoint $Rf_*$, defined on all $D(\sS(X))$ if $X$ has finite cohomological dimension and on $D^+(\sS(X))$ in general. Let $C\in D^b(\sS(Y))$ and $\sC\in D(\sS(X))$, bounded below unless $X$ has finite cohomological dimension. We have a ``projection morphism''
\begin{equation}\label{eq0}
Rf_*\sC\oo C\to Rf_*(\sC\oo f^*C)
\end{equation}
constructed as the adjoint of
\[f^*(Rf_*\sC\oo C)\iso f^*Rf_*\sC\oo  f^*C \by{\epsilon_\sC\oo 1}\sC \oo f^*C\]
where the first isomorphism is the inverse of the monoidal structure of $f^*$, and $\epsilon$ is the counit of the adjunction.

Suppose that $Y$ is the point and $f$ is the structural morphism, so that $\sS(Y)=\Ab$, the category of abelian groups. Then $Rf_*=R\Gamma(X,-)$, and

\begin{thm}\label{t1} The morphism \eqref{eq0} is an isomorphism in the following cases:
\begin{thlist}
\item $C$ is (isomorphic to) a perfect complex;
\item $R\Gamma$ commutes with infinite direct sums (e.g. $X$ is coherent \cite[VI]{sga}).
\end{thlist}
\end{thm}

\begin{proof} In Case (i) we reduce to $C=\Z[0]$ and the statement becomes trivial. Case (ii) follows from Case (i) as any bounded complex of abelian groups is a filtering colimit of perfect complexes, that we can write as a telescope in $D(\Ab)$ \cite{bn}.
\end{proof}

I don't know good conditions on a general $f$ to make such a proof work.\\

Now take the special case $C=A[0]$, $\sC=\sF[0]$. Theorem \ref{t1} gives an isomorphism
\[R\Gamma(X,\sF)\oo A\iso R\Gamma(X,\sF\oo A)\]
which is our universal coefficient theorem. Let us decipher it. 

Using the exact triangle $\Tor(\sF,A)[1]\to \sF\oo A\to (\sF\otimes A)[0]\by{+1}$, we get a long exact sequence
\begin{multline*}
\dots\to  H^{r+1}(X,\Tor(\sF,A))\to H^{r}(X,\sF\oo A)\to  H^{r}(X,\sF\otimes A)\\
 \to H^{r+2}(X,\Tor(\sF,A))\to\dots
\end{multline*}

On the other hand, there are split short exact sequences
\[0\to  H^{r}(X,\sF)\otimes A \to H^{r}(R\Gamma(X,\sF)\oo A)\to \Tor(H^{r+1}(X,\sF),A)\to 0 \]
given by the classical (!) universal coefficient theorem  \cite[Th. 3.6.2]{weibel}. Putting this together and using the ``lemma of the 700th'' \cite[Lemma p. 142]{mt}, we get 

\begin{cor}\label{c1} Under the following conditions:
\begin{thlist}
\item $A$ is finitely generated, or
\item $H^*$ commutes with filtering colimits (e.g. $X$ is coherent),
\end{thlist}
there are two complexes
\[ 0\to H^{r}(X,\sF)\otimes A\to H^{r}(X,\sF\otimes A)  \to H^{r+2}(X,\Tor(\sF,A))\]
\[H^{r-1}(X,\sF\otimes A)\to  H^{r+1}(X,\Tor(\sF,A))\to \Tor(H^{r+1}(X,\sF),A)\to 0  \]
with isomorphic homology. If $\Tor(\sF,A)=0$, we have split short exact sequences
\[0\to H^{r}(X,\sF)\otimes A\to H^{r}(X,\sF\otimes A)\to \Tor(H^{r+1}(X,\sF),A)\to 0\]
 recovering the classical formulation of the universal coefficient theorem.\qed
\end{cor}

\begin{cor}\label{c2} Under the conditions of Corollary \ref{c1}, assume that $H^{r+1}(X,\sF)$ is torsion-free and $H^{r+1}(X,\Tor(\sF,A))=0$. Then there is an exact sequence
\[ 0\to H^{r}(X,\sF)\otimes A\to H^{r}(X,\sF\otimes A)  \to H^{r+2}(X,\Tor(\sF,A)).\quad \square\]
\end{cor}

\end{document}